\documentclass[12pt]{article}
\usepackage[english]{babel}
\usepackage[letterpaper,top=2cm,bottom=2cm,left=2cm,right=2cm,marginparwidth=2cm]{geometry}

\usepackage[utf8]{inputenc}

\usepackage{amsmath}         
\usepackage{amsthm}  
\usepackage{amsfonts}
\usepackage{mathrsfs}
\usepackage{mathtools}
\usepackage{amssymb}

\renewcommand{\emptyset}{\varnothing} 

\newtheorem{theorem}{Theorem}[section]
\newtheorem{lemma}[theorem]{Lemma}
\newtheorem{proposition}[theorem]{Proposition}
\newtheorem{corollary}[theorem]{Corollary}

\theoremstyle{definition}
\newtheorem{definition}[theorem]{Definition}

\theoremstyle{remark}
\newtheorem*{remark}{Remark}

\usepackage{array}
\usepackage{makecell}

\usepackage{enumerate}

\usepackage{xcolor}
\definecolor{Green}{RGB}{69, 135, 51} 
\definecolor{Red}{RGB}{255, 0, 0}
\definecolor{gray}{RGB}{59, 59, 59} 
\newcommand{\Green}[1]{\textcolor{Green}{#1}}
\newcommand{\Red}[1]{\textcolor{Red}{#1}}
\usepackage[colorlinks=true, allcolors=blue]{hyperref}

\usepackage[normalem]{ulem}
\usepackage{csquotes}

\usepackage[symbol]{footmisc}
\renewcommand{\thefootnote}{\fnsymbol{footnote}}

\usepackage{comment}
\usepackage{framed}
\usepackage[style=3]{mdframed}
\mdfsetup{pstrickssetting={linestyle=dashed}}
\usepackage{verbatim}

\newcommand{\ZZ}{\mathbb{Z}} 

\newcommand{\oo}{\mathbf{o}}

\newcommand{\VV}{\mathbb{V}} 

\newcommand{\set}[1]{\left\{ #1 \right\}}
\newcommand{\abs}[1]{\left| #1 \right|}
\newcommand{\tup}[1]{\left( #1 \right)}

\newcommand{\Or}{\mathcal{O}}

\newcommand{\out}{\mathscr{O}}
\newcommand{\tu}{\mathscr{T}}

\usepackage{datetime}   
\newdateformat{monthyear}{\monthname[\THEMONTH] \THEYEAR}
\date{\monthyear\today}

\title{Degree Sequences vs. Forests in Finite Graphs}
\author{Benjamin Liber\footnote{Drexel University, Philadelphia, PA. \url{bl839@drexel.edu}}}
\date{\monthyear\today}

\begin{document}
\maketitle

\begin{abstract}
    We prove two conjectures of Shteiner and Shteyner stating that for an undirected graph $G=(V,E)$, the number of degree sequences arising from its spanning subgraphs is at least the number of forests in $G$, with equality if and only if $G$ is bipartite. In the process of proving the bipartite case, we provide several equivalent evaluations of the Tutte polynomial $T_G(x,y)$ at $(2,1)$, including interpretations in terms of degree vectors obtained from orientations of $G$. For the non-bipartite case, we prove strict inequality by expressing degree sequences as subset sums of signless incidence vectors and comparing these with linearly independent edge sets, showing that the presence of odd cycles yields additional independent sets beyond forests. We further strengthen this result by introducing odd pseudoforests, showing that their number is bounded above by the number of degree sequences and characterizing the corresponding independent sets accordingly.
\end{abstract}

\renewcommand{\thefootnote}{\arabic{footnote}}

In recent work on counting matrices \cite{shteiner2025comparingnumberssubforestssubgraphdegreetuples}, Shteiner and Shteyner have proposed two conjectures on the enumeration of degree sequences of subgraphs of a given graph.
Since different subgraphs can have the same degree sequence, this problem is nontrivial.
Shteiner and Shteyner have conjectured that the number of degree sequences of subgraphs of a bipartite graph $G$ equals the $(2,1)$-value of its Tutte polynomial (also known to be the number of forests in $G$),
and furthermore, that this equality becomes an inequality for non-bipartite $G$.

We prove both conjectures. In Section~\ref{sec:intro}, we introduce the relevant definitions and state the conjectures. In Section~\ref{sec:bipart}, we prove the first conjecture (the bipartite case) by developing several equivalent descriptions of the evaluation $T_{G}(2,1)$ of the Tutte polynomial, relating it to indegree, outdegree, and score vectors of orientations of a bipartite graph $G$. The argument relies on a related result for score vectors of orientations found by Stanley in 1980 \cite{Stanley1980,BrylawskiOxley1992,KleitmanWinston}, whose proof we include for the sake of completeness (following Brylawski and Oxley \cite{BrylawskiOxley1992} but adding more detail). In Section~\ref{sec.nbp}, we prove the second conjecture (the non-bipartite case) using an argument originating from the AI system called \url{bolzano.app} (see \cite{Bolzano26}) and shared with us by V\'aclav Rozho\v{n}.
At the core of the proof is an elegant lemma (Lemma~\ref{vector-list}) about finite lists of vectors, saying that the number of distinct sums of sublists of such a list is never smaller than the number of linearly independent sublists.
Although our contribution to these proofs is minimal, we believe that they deserve wider circulation.

Finally, in Section~\ref{sec.opf}, we extend the ideas of Section~\ref{sec.nbp} to obtain a strengthening of the main result in terms of odd pseudoforests. This direction was suggested by the same AI system. We develop it in full to show that the number of odd pseudoforests in a finite graph is bounded above by the number of degree sequences, and we also characterize the corresponding linearly independent sublists.

\section{Introduction}\label{sec:intro}

\begin{definition}\label{firstdef}
    Let $G=(V,E)$ be an undirected graph\footnote{We shall use standard terminology in graph theory, as described (e.g.) in \cite{BondyMurty} or \cite{Ruohonen}. All graphs in this note are also assumed to be \textbf{finite}.} with vertices $V = \{v_1,\ldots,v_n\}$, and let $H=(V,F)$ be a spanning\footnote{A \textbf{spanning} subgraph of $G$ means a subgraph that has the same vertex set as $G$.} subgraph of $G$, where $F \subseteq E$. We allow parallel edges and loops\footnote{A \textbf{loop} is an edge with only one endpoint. Note that loops contribute twice to vertex degrees.}.
    \begin{enumerate}
        \item The \textbf{degree sequence} (or degree function) of $H$ is the ordered $n$-tuple \begin{align*}
            \deg_H := (\deg_H(v_1),\ldots,\deg_H(v_n)),
        \end{align*} where $\deg_H(v_i)$ is the degree of $v_i$ in the graph $H$ (that is, the \#  of edges in $F$ incident to $v_i$, where loops are counted twice) for $i \in \{1,\ldots,n\}$.
        \item We say that $H$ is a \textbf{forest} in $G$ if it contains no cycles.
    \end{enumerate}
	For ease of notation, we will sometimes write an edge $e \in E$ connecting vertices $v_i,v_j \in V$ as $\{v_i,v_j\}$, even if there are several such edges.
\end{definition}

The goal of this note is to prove the following conjectures of Shteiner and Shteyner \cite[Hypothesis~2]{shteiner2025comparingnumberssubforestssubgraphdegreetuples}:

\begin{framed}
\begin{theorem}\label{maintheorem}
   Let $G=(V,E)$ be an undirected graph.
   \begin{enumerate}
       \item[(a)] If $G$ is bipartite, then \begin{align*}
        (\# \text{ of forests in $G$}) =(\# \text{ of degree sequences $\deg_H$ of spanning subgraphs $H \subseteq G$}).
    \end{align*}
    \item[(b)] If $G$ is non-bipartite (i.e., has an odd cycle\footnotemark), then \begin{align*}
        (\# \text{ of forests in $G$}) <(\# \text{ of degree sequences $\deg_H$ of spanning subgraphs $H \subseteq G$}).
        \end{align*}
   \end{enumerate}
\end{theorem} 
\end{framed} \footnotetext{A cycle is said to be \textbf{odd} if it has an odd number of edges. Otherwise, it is said to be \textbf{even}.}

It is important to note that distinct spanning subgraphs of a graph $G$ may have identical degree sequences. For example, if $G=C_4$, the cycle graph on four vertices, then the following two spanning subgraphs both have degree sequence $(1,1,1,1)$:
\begin{align*}
    H_1 := \Big(V, \big\{ \{v_1,v_2\},\{v_3,v_4\}\big\}\Big) \quad \text{and} \quad H_2 := \Big(V,\big\{\{v_1,v_4\},\{v_2,v_3\} \big\}\Big).
\end{align*}
Consequently, the number of spanning subgraphs of $G$ (which is $2^{\abs{E}}$) need not equal the number of distinct degree sequences of spanning subgraphs of $G$.

\begin{remark}
We have learned that Theorem~\ref{maintheorem}(a) has already been known to Stanley, appearing in his paper \cite[Corollary 5.4, last sentence]{Stanley1991}.
\end{remark}

\section{The Bipartite Case}\label{sec:bipart}

The proof of Theorem \ref{maintheorem}(a) will rely on several key results. We begin by introducing the Tutte polynomial $T_G(x,y)$ of a graph $G$, and we show that the number of forests in $G$ is given by evaluating this polynomial at the point $(2,1)$. Next, we discuss orientations of graphs, along with their associated indegree, outdegree, and score vectors. Finally, we establish a series of equalities connecting $T_G(2,1)$ to these vectors, which will form the backbone of the proof of Theorem \ref{maintheorem}(a).

\subsection{The Tutte Polynomial}

An edge $e$ of a graph $G$ is called a \textbf{bridge} if the removal of this edge from $G$ disconnects its two endpoints (i.e., its two endpoints belong to different connected components of the resulting graph).

The Tutte polynomial is a fundamental invariant of a graph, which is defined as follows:

\begin{definition}
    The \textbf{Tutte polynomial} of a graph $G=(V,E)$ is
    \begin{align*}
        T_G(x,y) = \sum_{A \subseteq E}(x-1)^{k(A)-k(E)}(y-1)^{k(A)+|A|-|V|}
        \in \ZZ[x,y]
    \end{align*}
    where $k(A)$ denotes the \textbf{number of connected components} of the graph $(V,A)$.
	Note that both exponents are nonnegative: this is clear for $k(A)-k(E)$, whereas for $k(A)+|A|-|V|$ it follows from Lemma~\ref{kE-ge} below.

    Alternatively, the Tutte polynomial can be defined recursively via a \textbf{deletion-contraction recurrence}, as seen in \cite[Theorem 6.2.2 and (6.13)]{BrylawskiOxley1992}:
	For any edge $e \in E$, we have
    \begin{align}
        T_G(x,y) = \begin{cases}
            T_{G-e}(x,y) + T_{G/e}(x,y), & \text{if $e$ is neither a loop nor a bridge;} \\
            xT_{G-e}(x,y), & \text{if $e$ is a bridge;} \\
            yT_{G-e}(x,y), & \text{if $e$ is a loop;} \\
            1, & \text{if $G$ has no edges,}
        \end{cases}
		\label{eq.tutte.dc}
    \end{align}
    where $G-e = (V,E\setminus \{e\})$ denotes the graph obtained from $G$ by \textbf{deleting} the edge $e$ and $G/e = (V(G/e),E(G/e))$ denotes the graph obtained from $G$ by \textbf{contracting} the edge $e$. That is, if $e$ joins vertices $v_i$ and $v_j$, then we merge the endpoints into a single vertex $v_{ij}$, and any edge incident to $v_i$ or $v_j$ in $G$ becomes incident to $v_{ij}$ in $G/e$. Formally, \begin{align*}
        V(G/e) = (V \setminus \{v_i,v_j\}) \cup \{v_{ij}\}.
    \end{align*}
\end{definition}

Many classical graph invariants are obtained as evaluations of $T_G$ at specific points $(x,y)$. Before stating the particular evaluation we will use in our proof of Theorem \ref{maintheorem}(a), let us note the following two well-known facts:

\begin{lemma}\label{kE-ge}
    For each graph $G=(V,E)$, we have $k(E) \ge |V|-|E|$.
\end{lemma}

\begin{proof}
    Let $G_1,\ldots,G_k$ denote the connected components of $G$, where $k=k(E)$, and write $G_i = (V_i,E_i)$ for all $i \in \{1,\ldots,k\}$. 

	Since each component of $G$ is connected, $|E_i| \geq |V_i|-1$ for all $i \in \{1,\ldots,k\}$, with equality if and only if $G_i$ is a tree. So,
	\begin{align}
        |E| = \sum_{i=1}^k|E_i| \geq \sum_{i=1}^k(|V_i|-1) = \Big(\sum_{i=1}^k|V_i|\Big)-k = |V|-k = |V|-k(E).
		\label{eq.pf.kE-ge}
    \end{align}
	In other words, $k(E) \ge |V|-|E|$.
\end{proof}

\begin{lemma}\label{forest_iff}
    A graph $G=(V,E)$ is a forest if and only if $k(E) = |V|-|E|$.
\end{lemma}

\begin{proof}
    Let $G_1,\ldots,G_k$ denote the connected components of $G$, where $k=k(E)$, and write $G_i = (V_i,E_i)$ for all $i \in \{1,\ldots,k\}$. 

    First, suppose $G$ is a forest. Then each connected component of $G$ is a tree, so that $|E_i| = |V_i|-1$ for $i \in \{1,\ldots,k\}$.
	Hence, the inequality \eqref{eq.pf.kE-ge} is an equality, since the $\ge$ sign was coming from $|E_i| \ge |V_i|-1$.
	In other words, $|E| = |V| - k(E)$, so that $k(E) = |V|-|E|$.

    Conversely, suppose $k(E) = |V|-|E|$.
	Then, $|E| = |V| - k(E)$, so that the inequality \eqref{eq.pf.kE-ge} is an equality.
	Hence, all the inequalities $|E_i| \ge |V_i|-1$ must be equalities as well (since \eqref{eq.pf.kE-ge} was obtained by summing these inequalities).
	In other words, $|E_i| = |V_i|-1$ for each $i$.
	Since each component $G_i = (V_i, E_i)$ of $G$ is connected, this entails that $G_i$ is a tree.
	Hence, each component of $G$ is a tree, so $G$ is a forest.
\end{proof}

As a consequence, we have the following evaluation of the Tutte polynomial:

\begin{proposition}
    \label{num_of_trees}
	Let $G = (V, E)$ be a graph.
	The number $T_G(2,1)$ is the number of forests in $G$.
\end{proposition}

\begin{proof}
    By the first definition of $T_G$, we have \begin{align}
        T_G(2,1) &= \sum_{A \subseteq E}(2-1)^{k(A)-k(E)}(1-1)^{k(A)+|A|-|V|} \nonumber\\
        &= \sum_{A \subseteq E}1^{k(A)-k(E)}0^{k(A)+|A|-|V|} \nonumber\\
        &= \sum_{A \subseteq E}0^{k(A)+|A|-|V|}\label{0inTG}.
    \end{align} By convention, $0^0 = 1$. So, \begin{align*}
        0^{k(A)+|A|-|V|} = \begin{cases}
            1, & \text{if } k(A)+|A|-|V|=0; \\
            0, & \text{otherwise.}
        \end{cases}
    \end{align*} Thus, \eqref{0inTG} reduces to \begin{align}
        T_G(2,1) = \sum_{\substack{A \subseteq E \text{ such that} \\ k(A)+|A|-|V|=0}}1 = \sum_{\substack{A \subseteq E \text{ such that} \\ k(A) = |V|-|A|}}1. \label{0inTG2}
    \end{align} By Lemma \ref{forest_iff}, $(V,A)$ is a forest if and only if $k(A) = |V|-|A|$. And so, \eqref{0inTG2} becomes \begin{align*}
        T_G(2,1) = \sum_{\substack{A \subseteq E \text{ such that} \\ (V,A) \text{ is a forest}}}1 = (\# \text{ of forests in $G$}). \tag*{\qedhere}
    \end{align*}
\end{proof}

\begin{corollary}\label{tgnotloop}
    Let $G = (V, E)$ be a graph.
    If $e \in E$ is not a loop, then \begin{align*}
        T_G(2,1) = T_{G-e}(2,1)+T_{G/e}(2,1).
    \end{align*}
\end{corollary}

\begin{proof}
    By (\ref{0inTG}), we have \begin{align*}
        T_{G}(2,1) = \sum_{A \subseteq E}0^{k(A)+|A|-|V|}.
    \end{align*} Hence, $T_G(2,1)$ does not depend on $k(E)$. Moreover, we can split this sum as follows: \begin{align*}
        T_G(2,1) = \sum_{\substack{A \subseteq E; \\ e \notin A}}0^{k(A)+|A|-|V|} + \sum_{\substack{A \subseteq E; \\ e \in A}}0^{k(A)+|A|-|V|}.
    \end{align*}

    Now, let us note that \begin{align*}
        \sum_{\substack{A \subseteq E; \\ e \notin A}}0^{k(A)+|A|-|V|} = \sum_{A \subseteq E\setminus \{e\}}0^{k(A)+|A|-|V|} = T_{G-e}(2,1)
    \end{align*}
	(by \eqref{0inTG}, now applied to the graph $G-e$).

    On the other hand, for any $A \subseteq E$ with $e \in A$, let $A' \subseteq E(G/e)$ denote the set $A \setminus \{e\}$, regarded as a set of edges of $G/e$ (so any edge containing an endpoint of $e$ will now contain the contracted vertex in $G/e$ instead).
	The map $A \mapsto A'$ is bijective because this process is reversible: every subset of $E(G/e)$ uniquely arises from a subset of $E$ containing $e$.

    Clearly, $|A'| = |A|-1$. Moreover, $k(A') = k(A)$, since any two vertices that are connected in $(V, A)$ are still connected in $(V(G/e), A')$ (as the removal of the edge $e$ is neutralized by the merging of its two endpoints) and vice versa (since $e \in A$ ensures that the two endpoints of $e$ are connected in $(V, A)$).
	Finally, since $e$ is not a loop, it joins two distinct vertices of $G$. Hence, contracting $e$ reduces the number of vertices by one, so that $|V(G/e)|=|V|-1$. Thus, \begin{align*}
        \underbrace{k(A')}_{\textcolor{gray}{=k(A)}}+\underbrace{|A'|}_{\textcolor{gray}{=|A|-1}}-\underbrace{|V(G/e)|}_{\textcolor{gray}{=|V|-1}} = k(A)+(|A|-1)-(|V|-1) = k(A)+|A|-|V|.
    \end{align*} So, reindexing the sum using the bijection $A \mapsto A'$, we see that \begin{align*}
        \sum_{\substack{A \subseteq E; \\ e \in A}}0^{k(A)+|A|-|V|} = \sum_{A' \subseteq E(G/e)}0^{k(A')+|A'|-|V(G/e)|} = T_{G/e}(2,1)
    \end{align*}
	(by \eqref{0inTG}, now applied to the graph $G/e$).
    
    Altogether, \begin{align*}
        T_{G}(2,1) &= \underbrace{\sum_{\substack{A \subseteq E; \\ e \notin A}}0^{k(A)+|A|-|V|}}_{\textcolor{gray}{=T_{G-e}(2,1)}}+\underbrace{\sum_{\substack{A \subseteq E; \\ e\in A}}0^{k(A)+|A|-|V|}}_{\textcolor{gray}{=T_{G/e}(2,1)}} = T_{G-e}(2,1)+T_{G/e}(2,1).
		\qedhere
    \end{align*}
\end{proof}

Thus, the value $T_G(2,1)$ of the Tutte polynomial satisfies a simpler form of the deletion-contraction recurrence \eqref{eq.tutte.dc}, namely:
\begin{align}
    T_{G}(2,1) = \begin{cases}
        T_{G-e}(2,1)+T_{G/e}(2,1), & \text{if $e$ is not a loop;} \\
        T_{G-e}(2,1), & \text{if $e$ is a loop}; \\
        1, & \text{if $G$ has no edges.}
    \end{cases}
    \label{eq.tutte21.dc}
\end{align}
(Indeed, the first case follows from Lemma~\ref{tgnotloop}; the second follows from Proposition~\ref{num_of_trees} since a forest cannot contain a loop; and the third is trivial.)

\subsection{Orientations of Graphs}

From now on, let us fix an undirected graph $G=(V,E)$ with vertices $V=\{v_1,\ldots,v_n\}$. 

\begin{definition}
    An \textbf{orientation} $\mathcal{O}$ of $G$ is a map
    \begin{align*}
        \mathcal{O} : E \to V \times V
    \end{align*}
    such that for each edge $e \in E$, the pair $\Or\tup{e}$ consists of the two endpoints of $e$ in some order (i.e., if $e$ has endpoints $v$ and $w$, then $\Or\tup{e}$ is either $\tup{v,w}$ or $\tup{w,v}$).
    That is, $\mathcal{O}$ is an assignment of direction to each edge of $G$.
    Such an orientation $\Or$ gives rise to a directed graph $G_{\mathcal{O}} := \tup{V,E}$, where each $e \in E$ is now viewed as an arc (directed edge) from $v$ to $w$, where $\Or\tup{e} = \tup{v,w}$.
    We say that an arc $(v,w)$ is directed \textbf{into} vertex $w$ and directed \textbf{out of} vertex $v$. We also call vertex $v$ the \textbf{tail} of the arc and vertex $w$ the \textbf{head}.
\end{definition}

\begin{definition} Fix an orientation $\mathcal{O}$ of $G$.
    \begin{enumerate}
        \item The \textbf{indegree vector} (or indegree function) of $G_{\mathcal{O}}$ is the tuple \begin{align*}
            \deg_{G_{\mathcal{O}}}^- := (\deg_{G_{\mathcal{O}}}^-(v_1),\ldots,\deg_{G_\mathcal{O}}^-(v_n)),
        \end{align*} where $\deg_{G_\mathcal{O}}^-(v_i) = (\# \text{ of arcs in $G_\mathcal{O}$ directed into $v_i$)}$ denotes the \textbf{indegree} of a vertex $v_i$.
        \item The \textbf{outdegree vector} (or outdegree function) of $G_\mathcal{O}$ is the tuple \begin{align*}
            \deg_{G_\mathcal{O}}^+ := (\deg_{G_\mathcal{O}}^+(v_1),\ldots,\deg_{G_\mathcal{O}}^+(v_n)),
        \end{align*} where $\deg_{G_\mathcal{O}}^+(v_i) = (\# \text{ of arcs in $G_\mathcal{O}$ directed out of $v_i$})$ denotes the \textbf{outdegree} of a vertex $v_i$.
        \item The \textbf{score vector} (or score function) of $G_\mathcal{O}$ is the tuple \begin{align*}
            s_{G_\mathcal{O}} := (s_{G_{\mathcal{O}}}(v_1),\ldots,s_{G_{\mathcal{O}}}(v_n)),
        \end{align*} where $s_{G_{\mathcal{O}}}(v_i) = \deg_{G_{\mathcal{O}}}^+(v_i) - \deg_{G_{\mathcal{O}}}^-(v_i)$ denotes the \textbf{score} of a vertex $v_i$. 
        \item For any spanning subgraph $H=(V,F)$ of $G$, where $F \subseteq E$, we define the spanning subdigraph
		\[
		H_{\mathcal{O}} := (V,\{\mathcal{O}(e) : e \in F\})
		\]
		of $G_{\Or}$. We define indegree vectors, outdegree vectors, and score vectors for any spanning subdigraph of $G_{\Or}$ in the same way as we defined them for $G_{\Or}$ itself.
		
    \end{enumerate}
\end{definition} Sometimes, we say ``$G$ has an outdegree (indegree, score) vector $(o_1,\ldots,o_n)$" to mean there exists an orientation $\mathcal{O}$ of $G$ such that the outdegree (indegree, score) vector of the oriented graph $G_{\mathcal{O}}$ is $(o_1,\ldots,o_n)$.
Similarly, we sometimes refer to the outdegree (indegree, score) vector of $G_{\mathcal{O}}$ as the ``outdegree (indegree, score) vector of $\mathcal{O}$''.

The following result, originally due to Brylawski and Oxley \cite[Lemma~6.3.20]{BrylawskiOxley1992}, is stated and proven here in a form suitable to our setting:

\begin{lemma}\label{outdeg}
    Suppose that both $(o_1,o_2,o_3,o_4,\ldots,o_n)$ and $(o_1',o_2',o_3,o_4,\ldots,o_n)$ are outdegree vectors of $G$ with $o_2 < o_2'$.
    Then $(o_1-1,o_2+1,o_3,o_4,\ldots,o_n)$ is also an outdegree vector of $G$.
\end{lemma}

\begin{proof}
    Let $\mathcal{O}$ and $\mathcal{O}'$ be two orientations of $G$ with outdegree vectors $(o_1,o_2,o_3,o_4,\ldots,o_n)$ and $(o_1',o_2',o_3,o_4,\ldots,o_n)$, respectively, and suppose $o_2 < o_2'$.

    Let us color the vertices $v_i \in V$ \Green{green} or \Red{red} by the following rule: a vertex $v_i \in V$ is colored \Red{red} if there exists a directed path in $G_{\mathcal{O}}$ from $v_1$ to $v_i$; otherwise, it is colored \Green{green}.
	Clearly, $v_1$ is \Red{red}.

	We say that an edge of $G$ is \Green{green}-\Green{green} if both of its endpoints are green.
	We also say that an edge of $G$ is \Green{green}-\Red{red} if its two endpoints have different colors.
	
	Note that if $(u, v)$ is an arc of $G_{\mathcal{O}}$ such that $u$ is \Red{red}, then $v$ must again be \Red{red} (the \Red{red}ness of $u$ means that there is a directed path in $G_{\mathcal{O}}$ from $v_1$ to $u$. So, we can concatenate this path with the arc $(u, v)$ to obtain a directed walk from $v_1$ to $v$; hence, there is also a directed path from $v_1$ to $v$, and thus, $v$ is also \Red{red}).
	In other words, an arc of $G_{\mathcal{O}}$ having a \Red{red} tail must also have a \Red{red} head.
	Consequently, any \Green{green}-\Red{red} edge of $G$ must have a \Green{green} tail in $G_{\mathcal{O}}$ (since otherwise, it would have a \Red{red} tail and thus -- according to the previous sentence -- also a \Red{red} head, contradicting the fact that it is a \Green{green}-\Red{red} edge).
    
	For any orientation $\mathcal{P}$ of $G$, let us set\footnote{Note that the colors \Green{green} and \Red{red} are still defined in terms of the orientation $\Or$, not $\mathcal{P}$.}
	\[
	S(\mathcal{P}) := \sum_{v_i \in V \text{ is \Green{green}}} \deg_{G_{\mathcal{P}}}^+(v_i).
	\]
	Thus, $S(\mathcal{P})$ is the sum of the outdegrees of all \Green{green} vertices in $G_{\mathcal{P}}$.
    In other words, $S(\mathcal{P})$ is the number of edges of $G$ whose tail in the oriented graph $G_{\mathcal{P}}$ is \Green{green}.
	These edges comprise
	\begin{enumerate}
	\item all the \Green{green}-\Green{green} edges of $G$, and
	\item those \Green{green}-\Red{red} edges of $G$ whose tail in $G_{\mathcal{P}}$ is \Green{green} (i.e., that are oriented from their \Green{green} endpoint to their \Red{red} one).
	\end{enumerate}
	In order to determine an orientation $\mathcal{P}$ that maximizes $S(\mathcal{P})$, we need to maximize the number of the latter edges (i.e., those \Green{green}-\Red{red} edges of $G$ whose tail in $G_{\mathcal{P}}$ is \Green{green}), since the former edges (i.e., the \Green{green}-\Green{green} edges of $G$) do not depend on $\mathcal{P}$.
	But clearly, the number of the latter edges is maximized for $\mathcal{P} = \Or$ (since we have shown that \textit{any} \Green{green}-\Red{red} edge of $G$ must have a \Green{green} tail in $G_{\mathcal{O}}$).
	Thus, the number $S(\mathcal{P})$ is maximized for $\mathcal{P} = \Or$. In other words,
	\[
	S(\mathcal{P}) \le S(\Or) \qquad \text{ for any orientation } \mathcal{P} \text{ of } G.
	\]
	In particular, $S(\mathcal{O}') \le S(\mathcal{O})$.
	
	Now, assume (for the sake of contradiction) that $v_2$ is \Green{green}.
    Recall that \begin{align*}
        \deg^+_{G_\mathcal{O}} = (o_1,o_2,o_3,o_4,\ldots,o_n) \quad \text{and} \quad \deg_{G_{\mathcal{O}'}}^+ = (o_1',o_2',o_3,o_4,\ldots,o_n).
    \end{align*}
	Thus, all vertices $v_i$ other than $v_1$ and $v_2$ have the same outdegree in $G_\mathcal{O}$ and $G_{\mathcal{O}'}$; that is, they satisfy $\deg_{G_{\mathcal{O}'}}^+(v_i) = \deg_{G_{\mathcal{O}}}^+(v_i)$.
	In particular, all \Green{green} vertices $v_i$ other than $v_2$ satisfy $\deg_{G_{\mathcal{O}'}}^+(v_i) = \deg_{G_{\mathcal{O}}}^+(v_i)$
	(since $v_1$ is \Red{red}).
	Hence, \begin{align*}
        S(\mathcal{O}') - S(\mathcal{O})
		=  \sum_{v_i \in V \text{ is \Green{green}}} \deg_{G_{\mathcal{O}'}}^+(v_i)
      		- \sum_{v_i \in V \text{ is \Green{green}}} \deg_{G_{\mathcal{O}}}^+(v_i)
		= o_2'-o_2
    \end{align*}
	(since the sums include the \Green{green} vertex $v_2$). 
	Therefore, $S(\mathcal{O}') - S(\mathcal{O}) =  o_2' - o_2 > 0$ (since $o_2 < o_2'$).
	So, $S(\mathcal{O}') > S(\mathcal{O})$, which contradicts $S(\mathcal{O}') \le S(\mathcal{O})$.
	Thus, our assumption was false, so that $v_2$ cannot be \Green{green}.
	
	Hence, $v_2$ is \Red{red}, meaning that there exists a directed path in $G_{\mathcal{O}}$ from $v_1$ to $v_2$. Reversing the direction of all edges in this directed path gives an orientation of $G$ having outdegree vector $(o_1-1,o_2+1,o_3,o_4,\ldots,o_n)$ (since $v_1$ loses one outgoing arc, $v_2$ gains an outgoing arc, and all intermediate vertices on the path trade one outgoing arc for another). This proves the lemma.
\end{proof}

\begin{corollary}\label{repeatcor}
    Let $e \in E$ join vertices $v_1$ and $v_2$.
	Let $\mathbf{o} := (o_3,o_4,\ldots,o_n) \in \mathbb{Z}^{n-2}$ be a vector such that the graph $G-e$ has an outdegree vector whose last $n-2$ entries equal $\mathbf{o}$ (meaning that they are $o_3,o_4,\ldots,o_n$ in this order).
	Then there exist integers $o_1,o_2$, and $k$ such that the set of all outdegree vectors of $G-e$ whose last $n-2$ entries equal $\mathbf{o}$ is \begin{align*}
        \{(o_1-j,o_2+j,o_3,\ldots,o_n) : 0 \leq j \leq k\}.
    \end{align*}
\end{corollary}

\begin{proof}
    For any orientation $\mathcal{O}$ of $G-e$, we have \begin{align}
        \sum_{i=1}^n\deg_{(G-e)_{\mathcal{O}}}^+(v_i) = |\operatorname{E}(G-e)|.
		\label{sum-deg+}
    \end{align} To be consistent with later notation, let $\mathscr{O}_{G-e}^\mathbf{o}$ denote the set of outdegree vectors of $G-e$ whose last $n-2$ entries equal $\mathbf{o}$. Since the last $n-2$ coordinates are fixed, the sum of the first two coordinates is constant amongst all vectors $(o_1, o_2, \ldots, o_n) \in \out_{G-e}^{\mathbf{o}}$, namely
		\begin{align*}
		o_1 + o_2 = |\operatorname{E}(G-e)| - \tup{o_3 + o_4 + \cdots + o_n} \qquad \textcolor{gray}{\tup{\text{by \eqref{sum-deg+}}}}.
		\end{align*}

		Let $(o_1^{\min},o_2^{\max},o_3,\ldots,o_n)$ and $(o_1^{\max},o_2^{\min},o_3,\ldots,o_n)$ be the vectors in $\out_{G-e}^{\mathbf{o}}$ with minimal and maximal first coordinate, respectively (noting that this in turn determines the second coordinates of each). Such vectors exist since $\out_{G-e}^\mathbf{o}$ is finite and nonempty (by the assumption that $G-e$ has an outdegree vector whose last $n-2$ entries equal $\mathbf{o}$).
		
		Starting from the vector $(o_1^{\max},o_2^{\min},o_3,\ldots,o_n)$ with maximal first coordinate (and therefore minimal second coordinate, due to $o_1+o_2$ being constant), we can repeatedly decrease the first coordinate by $1$ and simultaneously increase the second coordinate by $1$, until the first coordinate reaches $o_1^{\min}$.
		At each step, Lemma \ref{outdeg} guarantees the resulting vector is a valid outdegree vector of $G-e$. Doing this $k := o_1^{\max}-o_1^{\min}$ times produces the sequence \begin{align*}
			(o_1^{\max}-j,o_2^{\min}+j,o_3,\ldots,o_n), \quad j=0,1,\ldots,k,
		\end{align*} which starts at $(o_1^{\max},o_2^{\min},o_3,\ldots,o_n)$ and ends at $(o_1^{\min},o_2^{\max},o_3,\ldots,o_n)$. Every intermediate vector in this sequence is a valid outdegree vector, and by construction, these are all the vectors in $\out_{G-e}^{\mathbf{o}}$. Hence, \begin{align*}
			\out_{G-e}^{\mathbf{o}} = \{(o_1^{\max}-j,o_2^{\min}+j,o_3,\ldots,o_n) : j=0,1,\ldots,k\}. \tag*{\qedhere}
		\end{align*}
\end{proof}

\subsection{Equivalent Descriptions of $T_G\tup{2,1}$}

We now establish the following identities:

\begin{theorem}\label{equiv}
    We have
    \begin{align}
        T_G\tup{2,1}
        &= (\# \text{ of indegree vectors of orientations of $G$})
        \label{eq.equiv.12} \\
        &= (\# \text{ of outdegree vectors of orientations of $G$})
        \label{eq.equiv.13} \\
        &= (\# \text{ of score vectors of orientations of $G$}).
        \label{eq.equiv.14}
    \end{align}
    Moreover, for any given orientation $\Or$ of $G$, we have
    \begin{align}
        T_G\tup{2,1}
        &=
        (\# \text{ of score vectors of spanning subdigraphs of $G_{\Or}$}).
        \label{eq.equiv.15}
    \end{align}
\end{theorem}

To prove Theorem \ref{equiv}, we handle some equalities individually in the lemmas below before combining these results to establish the complete set.
We begin with the equality \eqref{eq.equiv.13}, which was originally discovered by Stanley in \cite{Stanley1980} and proved by Brylawski and Oxley in \cite[Proposition~6.3.19]{BrylawskiOxley1992} (see also \cite{KleitmanWinston} for a bijective proof\footnote{We note that the ``score vectors'' of \cite{KleitmanWinston} are not our score vectors, but rather our outdegree vectors.}).

\begin{lemma}\label{lem:1=3}
    We have $T_G(2,1) = (\# \text{ of outdegree vectors of orientations of $G$})$.
\end{lemma}

\begin{proof} 
    For any graph $H$ on $m$ vertices\footnote{Recall that $n$ is the number of vertices of our fixed graph $G$.}, let $\out_H \subseteq \mathbb{Z}^{m}$ denote the set of all outdegree vectors of $H$, and let $o(H) := |\out_H|$. Thus, our goal is to show that $T_G(2,1) = o(G)$.
	
	For any $m \geq n-2$, we define the map \begin{align*}
        \zeta : \mathbb{Z}^{m} \to \mathbb{Z}^{n-2}, \quad \mathbf{x}=(x_1,\ldots,x_m) \mapsto \zeta(\mathbf{x})=(x_{m-(n-3)},x_{m-(n-4)},\ldots,x_m);
    \end{align*} that is, $\zeta$ sends each vector $\mathbf{x} \in \mathbb{Z}^{m}$ to the vector of the last $n-2$ entries of $\mathbf{x}$. For any vector $\mathbf{o} \in \mathbb{Z}^{n-2}$, let \begin{align*}
        \mathscr{O}_H^{\mathbf{o}} :=
		\{\mathbf{x} \in \out_H : \zeta(\mathbf{x}) = \mathbf{o}\}
		=
		\{\text{outdegree vectors $\mathbf{x}$ of $H$ such that $\zeta(\mathbf{x}) = \mathbf{o}$}\} \subseteq \mathbb{Z}^{m},
    \end{align*} and set \begin{align*}
        \mathscr{T}_H := \zeta(\mathscr{O}_H) \subseteq \mathbb{Z}^{n-2}.
    \end{align*} Clearly, $\out_H^\textbf{o} \neq \emptyset$ if and only if $\textbf{o} \in \tu_H$. We can then partition \begin{align*}
        \out_H = \bigsqcup_{\textbf{o} \in \mathbb{Z}^{n-2}}\out_H^\textbf{o},
    \end{align*} so that \begin{align*}
        o(H) = |\out_H| = \sum_{\textbf{o} \in \mathbb{Z}^{n-2}}|\out_H^\textbf{o}|,
    \end{align*} noting that this sum is finite from our above remark.

    For our fixed graph $G$, we will show that $o(G)$ satisfies the same deletion-contraction recurrence \eqref{eq.tutte21.dc} as $T_G(2,1)$. Fix an edge $e \in E$.
    
    \bigskip
    \noindent\underline{Case One}: $e$ is not a loop. We will show that $o(G)=o(G-e)+o(G/e)$. Without loss of generality (we can permute the vertices at will), suppose $e$ joins vertices $v_1$ and $v_2$. 
    
    Fix a vector $\textbf{o} = (o_3,\ldots,o_n) \in \mathbb{Z}^{n-2}$.
	Suppose at first that $\oo \in \mathscr{T}_{G-e}$;
	that is, $G-e$ has an outdegree vector whose last $n-2$ entries are $\oo$.
	By Corollary~\ref{repeatcor}, there are integers $o_1,o_2$, and $k$ such that 
	\begin{align*}
        \out_{G-e}^{\textbf{o}} = \{(o_1-j,o_2+j,o_3,\ldots,o_n) : 0 \leq j \leq k\}.
    \end{align*}
	Hence, $\abs{\out_{G-e}^\textbf{o}} = k+1$.
	
	Now, given an orientation $\mathcal{O}$ of $G-e$, we can orient the edge $e$ in two different ways to obtain two orientations of $G$:
	\begin{enumerate}
        \item If $e$ is oriented as $(v_1,v_2)$, then the outdegree vector of the resulting orientation $\mathcal{O}_1$ satisfies \begin{align*}
            \deg_{G_{\mathcal{O}_1}}^+(v_1) = \deg_{(G-e)_{\mathcal{O}}}^+(v_1)+1 \quad \text{and} \quad \deg_{G_{\mathcal{O}_1}}^+(v_k) = \deg_{(G-e)_{\mathcal{O}}}^+(v_k) \ \forall k \neq 1.
        \end{align*} 
        \item If $e$ is oriented as $(v_2,v_1)$, then the outdegree vector of the resulting orientation $\mathcal{O}_2$ satisfies \begin{align*}
            \deg_{G_{\mathcal{O}_2}}^+(v_2) = \deg_{(G-e)_{\mathcal{O}}}^+(v_2)+1 \quad \text{and} \quad \deg_{G_{\mathcal{O}_2}}^+(v_k) = \deg_{(G-e)_{\mathcal{O}}}^+(v_k) \ \forall k \neq 2.
        \end{align*} 
    \end{enumerate}
	So, we see that each outdegree vector $(o_1-j,o_2+j,o_3,\ldots,o_n) \in \out_{G-e}^\textbf{o}$ yields two outdegree vectors in $\out_{G}^\textbf{o}$: \begin{align*}
        (o_1-j+1,o_2+j,o_3,\ldots,o_n) \quad \text{and} \quad (o_1-j,o_2+j+1,o_3,\ldots,o_n).
    \end{align*}
	This accounts for all outdegree vectors in $\out_{G}^\textbf{o}$ (since any orientation of $G$ can be transformed into an orientation of $G-e$ by removing the edge $e=\{v_1,v_2\}$, which decreases either the outdegree of $v_1$ or the outdegree of $v_2$ by $1$).
	Hence,
	\begin{align*}
        \out_{G}^{\textbf{o}}
		&= \{(o_1-j+1,o_2+j,o_3,\ldots,o_n) : 0 \leq j \leq k\}
		\cup \{(o_1-j,o_2+j+1,o_3,\ldots,o_n) : 0 \leq j \leq k\} \\
		&= \{(o_1-j+1,o_2+j,o_3,\ldots,o_n) : 0 \leq j \leq k+1\},
    \end{align*}
	since every $j \in \{0,1,\ldots,k\}$ satisfies $
        (o_1-(j+1)+1,o_2+(j+1),o_3,\ldots,o_n) = (o_1-j,o_2+j+1,o_3,\ldots,o_n)
    $.
	Therefore,
	$\abs{\out_G^\textbf{o}} = k+2 = \abs{\out_{G-e}^\textbf{o}}+1$ (since $\abs{\out_{G-e}^\textbf{o}} = k+1$).

	On the other hand, the graph $G/e$ has an orientation with outdegree vector $(o_1+o_2,o_3,\ldots,o_n)$, which can be obtained by taking any orientation of $G-e$ with outdegree vector $(o_1,o_2,o_3,\ldots,o_n)$ (this exists because $(o_1,o_2,o_3,\ldots,o_n) \in \out_{G-e}^\textbf{o}$) and contracting the edge $e$ (indeed, the contraction of $e$ identifies $v_1$ and $v_2$ into a single vertex whose outdegree is $o_1+o_2$, and all other vertices remain the same).
	Thus, $(o_1+o_2,o_3,\ldots,o_n) \in \out_{G/e}^\textbf{o}$.
	Moreover, $(o_1+o_2,o_3,\ldots,o_n)$ is the \textit{only} outdegree vector in $\out_{G/e}^\textbf{o}$ (since the sum of all outdegrees must be $\abs{E(G/e)} = \abs{E\setminus e} = o_1+o_2+\cdots+o_n$, but the last $n-2$ outdegrees must be the entries of $\mathbf{o}$).
	Hence, $|\out_{G/e}^\textbf{o}| = 1$.
	And so, $\abs{\out_G^\textbf{o}} = \abs{\out_{G-e}^\textbf{o}}+1 =|\out_{G-e}^\textbf{o}|+|\out_{G/e}^\textbf{o}|$.

    \bigskip
    On the other hand, suppose $\oo \notin \mathscr{T}_{G-e}$. Then $\out_{G-e}^\oo = \emptyset$. We will show that $\out_{G/e}^\oo = \emptyset$ and $\out_{G}^\oo = \emptyset$, for then $|\out_{G}^\textbf{o}| = 0 = |\out_{G-e}^\textbf{o}|+|\out_{G/e}^\textbf{o}|$.

    Suppose towards contradiction that $\textbf{o} \in \mathscr{T}_{G/e}$. Then there exists an orientation $\mathcal{O}'$ of $G/e$ such that $\zeta(\deg_{(G/e)_{\mathcal{O}'}}^+) = \oo$. Let $v_{12}$ denote the vertex obtained by merging $v_1$ and $v_2$ in $G/e$. By ``unmerging'' $v_{12}$ into $v_1$ and $v_2$, we obtain an orientation $\mathcal{O}$ of $G-e$ with the same outdegrees for vertices $v_3,\ldots,v_n$, so that $\zeta(\deg_{(G-e)_{\mathcal{O}}}^{+}) = \oo$. Hence, $\oo \in \mathscr{T}_{G-e}$, a contradiction. Thus, $\textbf{o} \notin \mathscr{T}_{G/e}$, so that $\mathscr{O}_{G/e}^\textbf{o} = \emptyset$.

    Similarly, if $\textbf{o} \in \mathscr{T}_G$, then removing the edge $e$ gives an orientation of $G-e$ with the same outdegrees for $v_3,\ldots,v_n$, again contradicting $\textbf{o} \notin \mathscr{T}_{G-e}$. Hence, $\textbf{o} \notin \mathscr{T}_{G}$, so that $\mathscr{O}_{G}^\textbf{o} = \emptyset$.
	
	\bigskip
	
	Thus, regardless of whether $\oo \in \mathscr{T}_{G-e}$ or $\oo \notin \mathscr{T}_{G-e}$, we have shown that
	\begin{align}
	|\out_{G}^\oo| = |\out_{G-e}^\oo|+|\out_{G/e}^\oo|.
	\label{eq.outoo}
	\end{align}
    
    \bigskip
    Altogether, since $\oo \in \mathbb{Z}^{n-2}$ was arbitrary, we have \begin{align*}
        o(G) &= |\out_G| \\
        &= \Big|\bigsqcup_{\textbf{o} \in \mathbb{Z}^{n-2}}\out_G^\textbf{o}\Big| \\
        &= \sum_{\textbf{o} \in \mathbb{Z}^{n-2}}|\out_G^\textbf{o}| \\
        &= \sum_{\textbf{o} \in \mathbb{Z}^{n-2}}\tup{|\out_{G-e}^\textbf{o}|+|\out_{G/e}^\textbf{o}|}  \qquad \tup{\text{by \eqref{eq.outoo}}} \\
        &= \sum_{\textbf{o} \in \mathbb{Z}^{n-2}}|\out_{G-e}^\textbf{o}|+\sum_{\textbf{o} \in \mathbb{Z}^{n-2}}|\out_{G/e}^\textbf{o}| \\
        &= \Big|\bigsqcup_{\textbf{o} \in \mathbb{Z}^{n-2}}\out_{G-e}^\textbf{o}\Big| + \Big|\bigsqcup_{\textbf{o} \in \mathbb{Z}^{n-2}}\out_{G/e}^\textbf{o}\Big| \\
        &= |\out_{G-e}|+|\out_{G/e}| \\
        &= o(G-e)+o(G/e).
    \end{align*}

    \bigskip
    \noindent\underline{Case Two}: $e$ is a loop. Without loss of generality, assume $e$ joins vertex $v_1$ to itself. Let $\mathcal{O}$ be any orientation of $G-e$ with outdegree vector $\deg_{(G-e)_{\mathcal{O}}}^+$. We can orient the edge $e$ in exactly one way, namely $(v_1,v_1)$, to obtain an orientation of $G$, call it $\mathcal{O}'$. The outdegree vector of $\mathcal{O}'$ satisfies the following: \begin{align*}
        \deg_{G_{\mathcal{O}'}}^+(v_1) = \deg_{(G-e)_{\mathcal{O}}}^+(v_1)+1, \quad
        \deg_{G_{\mathcal{O}'}}^+(v_k) = \deg_{(G-e)_{\mathcal{O}}}^+(v_k) \ \ \ \forall k \neq 1.
    \end{align*} So, each outdegree vector of $G-e$ corresponds to exactly one outdegree vector of $G$ in this manner. 
	This is a one-to-one correspondence. It follows that $o(G)=o(G-e)$. 

    \bigskip
    Finally, if $G$ has no edges, then it has exactly one outdegree vector, namely $(0,\ldots,0)$. 
    
    Altogether, we see that $o(G)$ follows the same recurrence (\ref{eq.tutte21.dc}) as $T_G(2,1)$: \begin{align*}
        o(G) = \begin{cases}
            o(G-e)+o(G/e), & \text{if $e$ is not a loop}; \\
            o(G-e), & \text{if $e$ is a loop}; \\
            1, & \text{if $G$ has no edges.}
        \end{cases} 
    \end{align*}
	Therefore, by induction on $|E|$, we obtain $o(G) = T_G(2,1)$, and the proof is complete.
\end{proof}

\begin{lemma}\label{lem:2=3}
    We have
    \begin{align*}
    & (\# \text{ of indegree vectors of orientations of $G$}) \\
    &= (\# \text{ of outdegree vectors of orientations of $G$}).
    \end{align*}
\end{lemma}

\begin{proof}
    For any orientation $\mathcal{O}$ of $G$ and any vertex $v_i \in V$, we have \begin{align*}
        \deg_{G_\mathcal{O}}^+(v_i) + \deg_{G_{\mathcal{O}}}^-(v_i) = \deg_G(v_i)
    \end{align*} (since each edge containing $v_i$ either has tail $v_i$ or has head $v_i$, and since loops are counted twice in the degree).
	Equivalently, \begin{align*}
        \deg_{G_\mathcal{O}}^+(v_i) = \deg_G(v_i) - \deg_{G_{\mathcal{O}}}^-(v_i).
    \end{align*} Hence, \begin{align}
        \deg_{G_{\mathcal{O}}}^+ = \deg_{G}-\deg_{G_{\mathcal{O}}}^-
		\label{eq.deg+=deg-deg-},
    \end{align} coordinate-wise. So, we have a bijective correspondence between indegree vectors and outdegree vectors: each indegree vector determines a unique outdegree vector, and each outdegree vector arises from exactly one indegree vector. Hence, \begin{align*}
        (\# \text{ of indegree vectors of orientations of $G$}) = (\# \text{ of outdegree vectors of orientations of $G$}).
    \end{align*} This proves the lemma. 
\end{proof}

\begin{lemma}\label{lem:3=4}
    We have
    \begin{align*}
    & (\#\text{ of outdegree vectors of orientations of $G$}) \\
    &= (\#\text{  of score vectors of orientations of $G$}).
    \end{align*}
\end{lemma}

\begin{proof}
    For any orientation $\mathcal{O}$ of $G$,  the definition of the score vector of $G_\mathcal{O}$ yields \begin{align}
        s_{G_\mathcal{O}} = \deg_{G_\mathcal{O}}^+-\deg_{G_\mathcal{O}}^-,
		\label{eq.s=deg+-deg-}
    \end{align} coordinate-wise. However, \eqref{eq.deg+=deg-deg-} yields $\deg_{G_\mathcal{O}}^- = \deg_G - \deg_{G_\mathcal{O}}^+$. Substituting this into \eqref{eq.s=deg+-deg-}, we get \begin{align*}
        s_{G_\mathcal{O}} = \deg_{G_\mathcal{O}}^+-\tup{\deg_G - \deg_{G_\mathcal{O}}^+}
		= 2\deg_{G_\mathcal{O}}^+-\deg_G.
    \end{align*} This is a bijective correspondence between outdegree vectors and score vectors: given $\deg_{G_\mathcal{O}}^+$, we can uniquely compute the corresponding score vector $s_{G_\mathcal{O}}$, and given $s_{G_\mathcal{O}}$, we can recover the corresponding outdegree vector as $\deg_{G_\mathcal{O}}^+ = \frac{1}{2}(s_{G_\mathcal{O}}+\deg_G)$. Hence, \begin{align*}
        (\# \text{ of outdegree vectors of orientations of $G$}) = (\# \text{ of score vectors of orientations of $G$}).\tag*{\qedhere}
    \end{align*}
\end{proof}

\begin{lemma}\label{lem:4=5}
     Fix an orientation $\mathcal{O}$ of $G$. Then,
     \[
     (\# \text{ of score vectors of orientations of $G$}) = (\# \text{ of score vectors of spanning subdigraphs of $G_{\Or}$}).
     \]
\end{lemma}

\begin{proof}
   Without loss of generality, we assume that $G$ has no loops, since all loops of $G$ can be removed without affecting either side of the lemma (since loops don't contribute to score vectors).
   For each subset $F \subseteq E$, let $\mathcal{O}^F$ denote the orientation of $G$ that differs from $\Or$ exactly in (the directions of) the arcs in $F$; that is, each arc $e \in F$ is assigned the direction in $\mathcal{O}^F$ opposite to its direction in $\mathcal{O}$, while all arcs in $E \setminus F$ retain the same direction in $\mathcal{O}^F$ as they did in $\mathcal{O}$.  Also, let $G_{\mathcal{O}}[F]$ denote the spanning subdigraph of $G_{\mathcal{O}}$ obtained by removing all arcs not in $F$.

    Recall that $s_D$ denotes the score vector of any digraph $D$. For any $F \subseteq E$, we can recover $s_{G_{\mathcal{O}}[F]}$ from $s_{G_{\mathcal{O}^F}}$ and vice-versa as follows: Fix a vertex $v_i$. Then \begin{align*}
        & s_{G_{\mathcal{O}^F}}(v_i) \\
		&= (\# \text{ arcs directed out of $v_i$ in $G_{\mathcal{O}^F}$}) - (\# \text{ arcs directed into $v_i$ in $G_{\mathcal{O}^F}$})
		\\
		&= \Big[(\# \text{ of arcs not in $F$ directed out of $v_i$ in $G_{\mathcal{O}}$}) + (\# \text{ of arcs in $F$ directed into $v_i$ in $G_{\mathcal{O}}$}) \Big] \\
        & \ \ \ \ - \Big[(\# \text{ of arcs not in $F$ directed into $v_i$ in $G_{\mathcal{O}}$}) + (\# \text{ of arcs in $F$ directed out of $v_i$ in $G_{\mathcal{O}}$})\Big] \\
		& \textcolor{gray}{\qquad \qquad \Big(\text{since the arcs not in $F$ are oriented the same way in $\Or^F$ as they are in $\Or$,}}\\
		& \qquad \qquad \quad \text{\textcolor{gray}{while the arcs in $F$ are oriented in opposite ways in $\Or^F$ and in $\Or$\Big)}} \\
        &= \Big[ (\# \text{ of arcs not in $F$ directed out of $v_i$ in $G_{\mathcal{O}}$}) - (\# \text{ of arcs not in $F$ directed into $v_i$ in $G_{\mathcal{O}}$)}\Big] \\
        & \ \ \ \ -\Big[ (\# \text{ of arcs in $F$ directed out of $v_i$ in $G_{\mathcal{O}}$}) - (\# \text{ of arcs in $F$ directed into $v_i$ in $G_{\mathcal{O}}$)}\Big] \\
        &= s_{G_{\mathcal{O}}[E \setminus F]}(v_i) - s_{G_{\mathcal{O}}[F]}(v_i).
    \end{align*}
    At the same time, \begin{align*}
        s_{G_{\mathcal{O}}[E\setminus F]}(v_i) + s_{G_{\mathcal{O}}[F]}(v_i) = s_{G_{\mathcal{O}}}(v_i),
    \end{align*} since each edge of $G$ lies in exactly one of $E \setminus F$ and $F$. 
	Eliminating $s_{G_{\mathcal{O}}[E\setminus F]}(v_i)$ from these two equations, we obtain
	\begin{align*}
        s_{G_{\mathcal{O}^F}}(v_i) &= s_{G_{\mathcal{O}}}(v_i) - 2s_{G_{\mathcal{O}}[F]}(v_i).
    \end{align*} Since $v_i \in V$ was arbitrary, we see that \begin{align*}
        s_{G_{\mathcal{O}^F}} = s_{G_{\mathcal{O}}}-2s_{G_{\mathcal{O}}[F]},
    \end{align*} component-wise, so that the vectors $s_{G_{\mathcal{O}^F}}$ and $s_{G_{\mathcal{O}}[F]}$ determine one another. Since every orientation $\Or'$ of $G$ is of the form $\mathcal{O}^F$ for some $F \subseteq E$ (namely, $F$ is the set of all edges of $G$ that are oriented differently in $\Or$ and $\Or'$),
	whereas every spanning subdigraph of $G_{\Or}$ has the form $G_{\Or}[F]$ for some $F \subseteq E$,
	this one-to-one correspondence reveals that \begin{align*}
        (\# \text{ of score vectors of orientations of $G$}) &= (\# \text{ of score vectors of spanning subdigraphs of $G_{\mathcal{O}}$}).
    \end{align*}
    This proves the lemma.
\end{proof} 

\begin{proof}[Proof of Theorem \ref{equiv}]
    This follows from Lemmas \ref{lem:1=3}, \ref{lem:2=3}, \ref{lem:3=4}, \ref{lem:4=5}.
\end{proof}

\subsection{Proof of Theorem \ref{maintheorem}(a)}
We can now prove Theorem \ref{maintheorem}(a):

\begin{proof}[Proof of Theorem \ref{maintheorem}(a)]
    Suppose $G$ is bipartite. Then, we can decompose $V = L \sqcup R$ such that each edge of $G$ has one endpoint in $L$ and the other in $R$. Orient all edges in $G$ from $L$ to $R$, i.e., each edge $e \in E$ joining vertices $v_i \in L$ and $v_j \in R$ is oriented $(v_i,v_j)$. Call this orientation $\mathcal{O}$.

    Let $H_{\mathcal{O}} \subseteq G_{\mathcal{O}}$ be any spanning subdigraph, and let $H$ be the corresponding underlying undirected subgraph of $G$ obtained by forgetting all orientations of arcs in $H_{\mathcal{O}}$. Note that $H$ and $H_{\Or}$ carry the same information, and indeed there is a bijection between the set of all spanning subdigraphs of $G_\Or$ and the set of all spanning subgraphs of $G$ that sends $H_\Or$ to $H$.
	
	Fix a vertex $v_i \in V$. 
    
    If $v_i \in L$, then all edges incident to $v_i$ are oriented outward; that is, $\deg_{H_\mathcal{O}}^-(v_i) = 0$ and $\deg_{H_{\mathcal{O}}}^+(v_i) = \deg_{H}(v_i)$. Hence, $s_{H_{\mathcal{O}}}(v_i) = \deg_{H_{\mathcal{O}}}^+(v_i)-\deg_{H_{\mathcal{O}}}^-(v_i)=\deg_{H}(v_i)$.
    
    On the other hand, if $v_i \in R$, then all edges incident to $v_i$ are oriented inward; that is, $\deg_{H_{\mathcal{O}}}^-(v_i) = \deg_{H}(v_i)$ and $\deg_{H_{\mathcal{O}}}^+(v_i) = 0$. Hence, $s_{H_{\mathcal{O}}}(v_i) = \deg_{H_{\mathcal{O}}}^+(v_i)-\deg_{H_{\mathcal{O}}}^-(v_i) = -\deg_{H}(v_i)$.
    Altogether, \begin{align*}
        s_{H_{\mathcal{O}}}(v_i) = \begin{cases}
            \deg_H(v_i), & \text{if } v_i \in L; \\
            -\deg_{H}(v_i), & \text{if } v_i \in R.
         \end{cases}
    \end{align*} Thus, we see that the score vector $s_{H_{\mathcal{O}}}$ is simply the degree sequence $\deg_{H}$ except with minus signs on all vertices $v_i\in R$. More precisely, there is a map \begin{align*}
        \Phi: \ & \set{\text{degree sequences of spanning subgraphs of $G$}} \\
		\to & \set{\text{score vectors of spanning subdigraphs of $G_{\Or}$}},
	\end{align*}
	defined by
	\begin{align*}
		(\Phi(\deg_H))(v_i) = \begin{cases}
            \deg_H(v_i), & \text{if } v_i \in L; \\
            -\deg_H(v_i), & \text{if } v_i \in R
        \end{cases} 
		\qquad \text{for all spanning subgraphs } H \subseteq G \text{ and all } v_i \in V.
    \end{align*} This map $\Phi$ is a bijection: It sends the degree sequence $\deg_H$ of any spanning subgraph $H \subseteq G$ to the score vector $s_{H_\Or}$ of the corresponding spanning subdigraph $H_{\mathcal{O}} \subseteq G_{\mathcal{O}}$, and every score vector arises uniquely from a degree sequence in this way. Hence, 
	\begin{align*}
        & (\# \text{ of degree sequences $\deg_H$ of spanning subgraphs $H \subseteq G$}) \\
        &= (\# \text{ of score vectors of spanning subdigraphs of $G_{\mathcal{O}}$}) \\
        &= T_{G}(2,1) \qquad \textcolor{gray}{\tup{\text{by \eqref{eq.equiv.15}}}} \\
		&= (\# \text{ of forests in $G$}) \qquad \textcolor{gray}{\tup{\text{by Proposition \ref{num_of_trees}}}}.
    \end{align*}
	This proves Theorem~\ref{maintheorem}(a).
\end{proof}

\section{\label{sec.nbp}The Non-Bipartite Case}

We now turn to the proof of Theorem~\ref{maintheorem}(b), which says that any non-bipartite graph $G$ satisfies \begin{align}\label{conj}
    (\# \text{ of forests in $G$}) <(\# \text{ of degree sequences $\deg_H$ of spanning subgraphs $H \subseteq G$}).
\end{align}

Shteiner and Shteyner proved inequality (\ref{conj}) for odd cycles \cite[Proposition~5.10]{shteiner2025comparingnumberssubforestssubgraphdegreetuples}, cactus graphs \cite[Theorem~5.22]{shteiner2025comparingnumberssubforestssubgraphdegreetuples}, and generalized book graphs \cite[Theorem~5.28]{shteiner2025comparingnumberssubforestssubgraphdegreetuples}. Furthermore, if $G$ is an \emph{FHM-graph} (meaning that no induced subgraph of $G$ consists of two vertex-disjoint odd-length cycles and no further edges), then \eqref{conj} follows from \cite[Theorem 5.3 and Theorem 5.1]{Stanley1991}.

We will now prove inequality (\ref{conj}) for an arbitrary non-bipartite graph. To do so, we interpret degree sequences as subset sums of signless incidence vectors and compare these with linearly independent families of such vectors.
The following results were found by the AI system \url{bolzano.app} \cite{Bolzano26}, but the writeup is our own.

\subsection{\label{subsec.vecs}Subset Sums and Linearly Independent Index Sets}

Let $\VV$ be an $\mathbb{R}$-vector space, and let $X = \tup{\mathbf{x}_1,\ldots,\mathbf{x}_m}$ be an (ordered) list of vectors $\mathbf{x}_i \in \VV$.
Set $[m] := \{1,2,\ldots,m\}$. We take \begin{align*}
    \Sigma(X) := \set{\sum_{i \in A}\mathbf{x}_i : A \subseteq [m]}
\end{align*} to denote the \textbf{set of distinct subset sums} of $X$, and we take \begin{align*}
    \mathcal{I}(X) := \set{A \subseteq [m] : \tup{\mathbf{x}_i}_{i \in A} \text{ is linearly independent in } \VV} 
\end{align*} to denote the \textbf{family of linearly independent index sets} of $X$.

\begin{lemma}\label{vector-list}
    For every finite list $X$ of real vectors, \begin{align}\label{subset_sums_index_sets}
        \abs{\Sigma(X)} \geq \abs{\mathcal{I}(X)}.
    \end{align}
\end{lemma}

\begin{proof}
    Let us induct on $m=\abs{X}$.
    \medskip
    
    \textit{Base case:} The case $m=0$ is immediate: In this case, $\Sigma(X) = \set{\mathbf{0}}$ and $\mathcal{I}\tup{X} = \set{\emptyset}$ (since the empty list of vectors is linearly independent), so that both sides of \eqref{subset_sums_index_sets} equal $1$. This completes the base case.


    \medskip
    \textit{Induction step:} Next, suppose \eqref{subset_sums_index_sets} holds for all lists of $m-1$ vectors in a real vector space. Let $X = \tup{\mathbf{x}_1,\ldots,\mathbf{x}_{m-1},\mathbf{x}_m}$ be a list of $m$ vectors in $\VV$. Set \begin{align*}
        X' := \tup{\mathbf{x}_1,\ldots,\mathbf{x}_{m-1}}, \quad \mathbf{x} := \mathbf{x}_m,
    \end{align*} so that $X = \tup{X',\mathbf{x}}$ and $|X'|=m-1$.

    First, suppose $\mathbf{x} = \mathbf{0}$. Then \eqref{subset_sums_index_sets} follows immediately: adjoining the zero vector does not create any new subset sums, and no linearly independent subset of $X$ contains $\mathbf{0}$. Hence, \begin{align*}
        \Sigma(X) = \Sigma(X') \quad \text{and} \quad \mathcal{I}(X) = \mathcal{I}(X'),
    \end{align*} so the result follows from the induction hypothesis applied to $X'$.

    And so, we may assume that $\mathbf{x} \neq \mathbf{0}$. Let $S := \Sigma(X')$. By the induction hypothesis, \begin{align}\label{induct_X}
        \abs{S} = \abs{\Sigma(X')} \geq \abs{\mathcal{I}(X')}. 
\end{align} Moreover, \[\Sigma(X) = S \cup (S+\mathbf{x}),\] where $S+\mathbf{x} = \set{\mathbf{s}+\mathbf{x} : \mathbf{s} \in S}$. Indeed, every subset sum either excludes $\mathbf{x}$, yielding an element of $S$, or includes $\mathbf{x}$, yielding an element of $S+\mathbf{x}$. Hence, \begin{align}\label{split_sigmaX}
    \abs{\Sigma(X)} = \abs{S} + \abs{(S+\mathbf{x}) \setminus S}.
\end{align}

Next, consider the $1$-dimensional vector subspace \begin{align*}
    \langle \mathbf{x} \rangle = \set{c\mathbf{x} : c \in \mathbb{R}}
\end{align*} of $\VV$, and let $\pi : \VV \to \VV / \langle \mathbf{x} \rangle$ denote the quotient map, given by \begin{align*}
    \pi(\mathbf{v}) = \mathbf{v}+ \langle \mathbf{x} \rangle \quad \text{for all } \mathbf{v} \in \VV.
\end{align*} In particular, since $\mathbf{x} \in \langle \mathbf{x} \rangle$, we have $\pi(\mathbf{x}) = \mathbf{0}$. 

Let $\ell : \VV \to \mathbb{R}$ be a linear functional such that $\ell(\mathbf{x}) > 0$ (since $\mathbf{x} \neq \mathbf{0}$, such a linear functional exists). For each $y \in \pi(S)$, let \begin{align*}
    F_y := S \cap \pi^{-1}(y),
\end{align*} which is finite and nonempty. Choose $\mathbf{s}_y \in F_y$ such that \begin{align*}
    \ell(\mathbf{s}_y) = \max_{\mathbf{s} \in F_y}\ell(\mathbf{s}).
\end{align*} Since $\pi(\mathbf{x}) = 0$, we have \begin{align*}
    \pi(\mathbf{s}_y + \mathbf{x}) = \pi(\mathbf{s}_y) = y
    \qquad \textcolor{gray}{\left(\text{since } \mathbf{s}_y \in F_y = S \cap \pi^{-1}(y) \subseteq \pi^{-1}(y)\right)}.
\end{align*} Moreover, since $\ell$ is linear and $\ell(\mathbf{x}) > 0$, it follows that \begin{align*}
    \ell(\mathbf{s}_y + \mathbf{x}) = \ell(\mathbf{s}_y) + \ell(\mathbf{x}) > \ell(\mathbf{s}_y).
\end{align*} Hence, by the maximality of $\mathbf{s}_y$ in $F_y$, we have $\mathbf{s}_y + \mathbf{x} \notin S$. That is, \begin{align*}
    \mathbf{s}_y + \mathbf{x} \in (S+\mathbf{x}) \setminus S.
\end{align*} Moreover, if $y_1,y_2 \in \pi(S)$ are such that $y_1 \neq y_2$, then $\pi(\mathbf{s}_{y_1}) = y_1 \neq y_2 = \pi(\mathbf{s}_{y_2})$ and therefore $\mathbf{s}_{y_1} \neq \mathbf{s}_{y_2}$, so that $\mathbf{s}_{y_1} + \mathbf{x} \neq \mathbf{s}_{y_2} + \mathbf{x}$ as well.
And so, \begin{align*}
    y \mapsto \mathbf{s}_y + \mathbf{x}
\end{align*} defines an injection from $\pi(S)$ into $(S+\mathbf{x}) \setminus S$; consequently, \begin{align}\label{injection}
    \abs{(S+\mathbf{x}) \setminus S} \geq \abs{\pi(S)}.
\end{align} 

Next, by linearity of $\pi$, \begin{align*}
    \pi(S) &= \pi\tup{\Sigma(X')} = \set{\pi \tup{\sum_{i \in A}\mathbf{x}_i } : A \subseteq [m-1]} \\ &= \set{\sum_{i \in A}\pi(\mathbf{x}_i) : A \subseteq [m-1]} = \Sigma\tup{\pi(X')}.
\end{align*} We thus obtain \begin{align}\label{induct_X'}
    |\pi(S)| = \abs{\Sigma(\pi(X'))} \geq \abs{\mathcal{I}(\pi(X'))}
\end{align}
by applying the induction hypothesis to $\pi(X')$.
From \eqref{injection} and \eqref{induct_X'}, we obtain \begin{align}\label{ineq_S+x-S}
    \abs{(S+\mathbf{x}) \setminus S} \geq \abs{\pi(S)} \geq \abs{\mathcal{I}(\pi(X'))}.
\end{align} And so, \eqref{split_sigmaX} yields \begin{align}
    \abs{\Sigma(X)} &= \underbrace{\abs{S}}_{\textcolor{gray}{\substack{\geq \abs{\mathcal{I}(X')} \\ \text{(by (\ref{induct_X}))}}}}+ \ \underbrace{\abs{\tup{S + \mathbf{x}} \setminus S}}_{\textcolor{gray}{\substack{\geq \abs{\mathcal{I}(\pi(X'))} \\ \text{(by (\ref{ineq_S+x-S}))}}}} \qquad  \nonumber \\
    &\geq \abs{\mathcal{I}(X')} + \abs{\mathcal{I}(\pi(X'))}.\label{splitI} 
\end{align} 

Finally, let us identify $\mathcal{I}(X)$ in terms of $\mathcal{I}(X')$ and $\mathcal{I}(\pi(X'))$. Every subset $A \subseteq [m]$ either contains $m$ or it does not, and these two cases are disjoint. 

If $m \notin A$, then $A \subseteq [m-1]$; hence, $A \in \mathcal{I}(X)$ if and only if $A \in \mathcal{I}(X')$. 
Thus, the sets $A \in \mathcal{I}\tup{X}$ that don't contain $m$ are exactly the sets $A \in \mathcal{I}\tup{X'}$.

If $m \in A$, let us write $A = B \cup \{m\}$, where $B \subseteq [m-1]$. Then $\tup{\mathbf{x}_i : i \in A}$ is linearly independent if and only if $\tup{\pi(\mathbf{x}_i) : i \in B}$ is linearly independent\footnote{Indeed, if $\sum_{i \in B}c_i\pi(\mathbf{x}_i) = \mathbf{0}$, then linearity of $\pi$ yields $\pi\tup{\sum_{i \in B}c_i\mathbf{x}_i} = \mathbf{0}$. Hence, $\sum_{i \in B}c_i\mathbf{x}_i \in \ker(\pi) = \langle \mathbf{x} \rangle = \langle \mathbf{x}_m \rangle$. Thus, there exists $c_m \in \mathbb{R}$ such that $\sum_{i \in B}c_i\mathbf{x}_i +c_m\mathbf{x}_m = \mathbf{0}$. Conversely, any linear dependence relation among $\tup{\mathbf{x}_i : i \in B \cup \{m\} = A}$ projects under $\pi$ to a linear dependence relation among $\tup{\pi(\mathbf{x}_i) : i \in B}$, since $\mathbf{x}_m = \mathbf{x} \neq \mathbf{0}$ ensures that the $\mathbf{x}_m$-coefficient cannot be the only nonzero coefficient.}, meaning that $B \in \mathcal{I}\tup{\pi(X')}$.
In other words, $A \in \mathcal{I}\tup{X}$ if and only if $B \in \mathcal{I}\tup{\pi(X')}$.
Thus, the sets $A \in \mathcal{I}\tup{X}$ that contain $m$ are in bijection with the sets $B \in \mathcal{I}\tup{\pi(X')}$.

Combining the results of both cases, we find
\begin{align*}
    \mathcal{I}(X) = \underbrace{\mathcal{I}(X')}_{\textcolor{gray}{\text{sets $A$ not containing $m$}}} \sqcup\ \ \ \ \underbrace{\set{B \cup \{m\} : B \in \mathcal{I}(\pi(X'))}}_{\textcolor{gray}{\text{sets $A$ containing $m$}}},
\end{align*} so that
\begin{align*}
    \abs{\mathcal{I}(X)} = \abs{\mathcal{I}(X')} + \abs{\mathcal{I}(\pi(X'))}.
\end{align*} And thus, from (\ref{splitI}), we obtain \begin{align*}
    \abs{\Sigma(X)} \geq \abs{\mathcal{I}(X')} + \abs{\mathcal{I}(\pi(X'))} = \abs{\mathcal{I}(X)}.
\end{align*} This proves (\ref{subset_sums_index_sets}).
\end{proof}

\subsection{Signless Incidence Vectors}
Let us turn our attention back to graphs. Let $G = (V,E)$ be an arbitrary graph with vertices $V=\set{v_1,\ldots,v_n}$. For an edge $e \in E$ with endpoints $v_i$ and $v_j$ (not necessarily distinct), we define its \textbf{signless incidence vector} as the vector \begin{align*}
    \mathbf{b}_e := \mathbf{e}_{v_i}+\mathbf{e}_{v_j}
    \in \mathbb{R}^V
\end{align*} where $\mathbf{e}_{v_i}$ denotes the standard basis vector of $\mathbb{R}^V$ corresponding to vertex $v_i$.

Recall from Definition~\ref{firstdef}(1) that the degree sequence of a spanning subgraph $H=(V,F)$ of $G$ is the ordered $n$-tuple \begin{align*}
    \deg_H = \tup{\deg_H(v_1),\ldots,\deg_H(v_n)}.
\end{align*} When convenient, we identify this tuple with the vector \begin{align*}
    \sum_{i=1}^n\deg_H(v_i)\mathbf{e}_{v_i} \in \mathbb{R}^V.
\end{align*}

\begin{lemma}\label{forest_incidence}
    For a spanning subgraph $H=(V,F)$ of $G$, we have \begin{align}\label{degsec_vec}
    \deg_H = \sum_{e \in F}\mathbf{b}_e.
\end{align}
\end{lemma}

\begin{proof}
    By definition of degree sequence (viewed as a vector in $\mathbb{R}^V$), the $v_i$-coordinate of $\deg_H$ is $\deg_H(v_i)$, which counts the number of edges in $F$ incident to $v_i$ (with each loop at $v_i$ contributing twice).

    On the other hand, consider the right-hand side of \eqref{degsec_vec}. Fix a vertex $v_i \in V$, and examine its $v_i$-coordinate. Each edge of $H$ contributes to this coordinate as follows. If $e = \set{v_i,v_j} \in F$ with $j \neq i$, then $\mathbf{b}_e = \mathbf{e}_{v_i}+\mathbf{e}_{v_j}$, so this edge contributes $1$ to the $v_i$-coordinate. If $e = \set{v_i,v_i} \in F$ (that is, $e$ is a loop with endpoint $v_i$), then $\mathbf{b}_e = 2\mathbf{e}_{v_i}$, so this edge contributes $2$ to the $v_i$-coordinate. Finally, if $e \in F$ is not incident to $v_i$, then it contributes $0$ to the $v_i$-coordinate.

    And so, the $v_i$-coordinate of $\sum_{e \in F}\mathbf{b}_e$ is exactly the number of edges of $H$ incident to $v_i$, counting loops twice; this is precisely $\deg_H(v_i)$. Since this holds for every $i$, the two vectors agree in every coordinate; thus, the lemma is proven.
\end{proof}

Let $X_G := \tup{\mathbf{b}_e : e \in E}$, viewed as an ordered (labeled) list. Then, by Lemma \ref{forest_incidence}, we have \begin{align}\label{sum_equals_degs}
    \Sigma(X_G) = \set{\sum_{e \in F}\mathbf{b}_e : F \subseteq E}
    = \set{\deg_H : H=(V,F) \text{ is a spanning subgraph of $G$}}
\end{align} (using the notations of Subsection \ref{subsec.vecs});
that is, the degree sequences of spanning subgraphs of $G$ are precisely the elements of $\Sigma(X_G)$.

Next, the following two lemmas identify some edge sets of $G$ whose corresponding signless incidence vectors are linearly independent:

\begin{lemma}\label{forests_independence}
    If $F \subseteq E$ is the edge set of a forest in $G$, then the corresponding signless incidence vectors $\tup{\mathbf{b}_e : e \in F}$ are linearly independent in $\mathbb{R}^V$.
\end{lemma}

\begin{proof}
    Suppose \begin{align}\label{linin_forest}
        \sum_{e \in F}a_e\mathbf{b}_e = \mathbf{0}
    \end{align} for some scalars $a_e \in \mathbb{R}$. We will show that $a_e = 0$ for all $e \in F$.

    We proceed by induction on $|F|$. The base case $|F| = 0$ is immediate. Next, suppose $|F| = k \geq 1$, and assume that every forest with fewer than $k$ edges has linearly independent signless incidence vectors. Since $F$ is a nonempty forest, it contains a leaf vertex $v$ incident to a unique edge $e_0 \in F$. Considering the $v$-coordinate of \eqref{linin_forest}, every edge $e \in F \setminus \{e_0\}$ contributes zero (since no such edge is incident to $v$, and thus the $v$-coordinate of $\mathbf{b}_e$ is zero for such an edge). Since the $v$-coordinate of $\mathbf{b}_{e_0}$ is $1$, it follows that
    the $v$-coordinate of the left-hand side of \eqref{linin_forest} equals $a_{e_0}$. Therefore, $a_{e_0} = 0$.

    Finally, let $F' = F \setminus \set{e_0}$. Since every subgraph of a forest is a forest, $F'$ is itself the edge set of a forest with $|F'| < k$. Since $a_{e_0} = 0$, equation \eqref{linin_forest} reduces to \begin{align*}
        \sum_{e \in F'}a_e\mathbf{b}_e = \mathbf{0}.
    \end{align*} By the induction hypothesis, it follows that $a_e = 0$ for all $e \in F'$. Together with $a_{e_0} = 0$, we conclude that $a_e = 0$ for all $e \in F$. Hence, $\tup{\mathbf{b}_e : e \in F}$ is linearly independent in $\mathbb{R}^V$.
\end{proof}

\begin{lemma}\label{odd_cycles_independence}
    If $C \subseteq E$ is the edge set of an odd cycle in $G$, then the corresponding signless incidence vectors $\tup{\mathbf{b}_e : e \in C}$ are linearly independent in $\mathbb{R}^V$.
\end{lemma}

\begin{proof}
    Suppose $|C|=2k+1$ for some $k \geq 0$. Label the vertices of the cycle $w_0,w_1,\ldots,w_{2k}$. Let $e_i \in C$ join vertices $w_{i}$ and $w_{i+1}$ for $i=0,1,\ldots,2k-1$,  and let $e_{2k} \in C$ join vertices $w_{2k}$ and $w_0$.

   Suppose \begin{align}\label{linin_cycle}
       \sum_{i=0}^{2k}c_i\mathbf{b}_{e_{i}} = \mathbf{0}
   \end{align} for some scalars $c_i \in \mathbb{R}$. We will show that $c_i = 0$ for all $i \in \{0,1,\ldots,2k\}$.

   Fix a vertex $w_i$ in the cycle. The only edges of $C$ incident to $w_i$ are $e_{i-1}$ and $e_i$ (with $e_{-1} := e_{2k}$), and each contributes $1$ to the $w_i$-coordinate of \eqref{linin_cycle}. So, we obtain $c_{i-1}+c_i=0$ for all $i$; that is, \begin{align}\label{c_negation}
       c_{i} = -c_{i-1} \quad \text{for all } i
   \end{align} (where $c_{-1} := c_{2k}$).
   Iterating this relation gives \begin{align}
       \label{c_reduction}
       c_i = (-1)^i c_0 \quad \text{for all } i.
   \end{align} In particular, \begin{align*}
       c_{2k} = (-1)^{2k}c_0 = c_0.
   \end{align*}
   On the other hand, relation \eqref{c_negation} (for $i = 0$) gives $c_0 = - c_{-1} = - c_{2k} = - c_0$, so that $c_0 = 0$. Hence, \eqref{c_reduction} shows that $c_i = 0$ for all $i$. Thus, $\tup{\mathbf{b}_e : e \in C}$ is linearly independent in $\mathbb{R}^V$.
\end{proof}

\subsection{Proof of Theorem \ref{maintheorem}(b)}

We can now prove Theorem \ref{maintheorem}(b):

\begin{proof}[Proof of Theorem \ref{maintheorem}(b)] Suppose $G=(V,E)$ is non-bipartite. Then $G$ contains an odd cycle with edge set $C \subseteq E$. By Lemma \ref{odd_cycles_independence}, the corresponding family of signless incidence vectors $\tup{\mathbf{b}_e : e \in C}$ is linearly independent. Hence, $C \in \mathcal{I}(X_G)$.

On the other hand, Lemma~\ref{forests_independence} implies that every edge set of a forest lies in $\mathcal{I}(X_G)$. Hence, \begin{align*}
    \mathcal{I}(X_G) \supseteq \set{F \subseteq E : F \text{ is the edge set of a forest}}.
\end{align*} Since $C$ is the edge set of a cycle, it is not the edge set of a forest. So, we may conclude that \begin{align*}
    \mathcal{I}(X_G) \supsetneq \set{F \subseteq E : F \text{ is the edge set of a forest}};
\end{align*} consequently, \begin{align}\label{ind_forests}
    \abs{\mathcal{I}(X_G)}
    > \abs{\set{F \subseteq E : F \text{ is the edge set of a forest}}}
    = \tup{\# \text{ of forests in $G$}}
\end{align}
(since forests in $G$ are uniquely determined by their edge sets).

And so, we have \begin{align*}
    &(\# \text{ of degree sequences $\deg_H$ of spanning subgraphs $H \subseteq G$}) \\ &= \abs{\Sigma(X_G)} \quad \textcolor{gray}{\text{(by \eqref{sum_equals_degs})}} \\
    & \geq \abs{\mathcal{I}(X_G)} \quad \textcolor{gray}{\text{(by Lemma~\ref{vector-list})}} \\
    & > (\# \text{ of forests in $G$}) \quad \textcolor{gray}{(\text{by \eqref{ind_forests})}}.  
\end{align*} This proves Theorem~\ref{maintheorem}(b).
    
\end{proof}

\section{\label{sec.opf}Extension to Odd Pseudoforests}

\subsection{Strengthening of Theorem~\ref{maintheorem}}

We now strengthen Theorem \ref{maintheorem} by replacing forests with odd pseudoforests, as defined below:

\begin{definition}
    Let $G=(V,E)$ be an undirected graph, and let $H=(V,F)$ be a spanning subgraph of $G$, where $F \subseteq E$. We allow parallel edges and loops.
    \begin{enumerate}
        \item We say that $H$ is a \textbf{pseudoforest} in $G$ if every connected component of $H$ contains at most one cycle\footnote{Cycles that differ only by a cyclic rotation are considered to be identical.}.
        \item We say that $H$ is an \textbf{odd pseudoforest} in $G$ if $H$ is a pseudoforest in which every cycle is odd.
    \end{enumerate}
\end{definition}

The following result was found by \url{bolzano.app} \cite{Bolzano26} with a brief intuitive explanation as to why it holds. We provide a rigorous proof below: 
\begin{theorem}\label{pseudoforests_theorem}
    Let $G=(V,E)$ be an undirected graph. Then, \begin{align*}
        \tup{\# \text{ of odd pseudoforests in $G$}} \leq \tup{\# \text{ of degree sequences $\deg_H$ of spanning subgraphs $H \subseteq G$}}.
    \end{align*}
\end{theorem}

For bipartite graphs, odd pseudoforests coincide with forests, so Theorem~\ref{pseudoforests_theorem} gives one direction of Theorem~\ref{maintheorem}(a).
In contrast, for non-bipartite graphs, the presence of odd cycles yields the strict inequality in Theorem~\ref{maintheorem}(b).

In order to prove Theorem~\ref{pseudoforests_theorem}, we will need the following four lemmas:

\begin{lemma}\label{two_leaves_path}
   Let $G=(V,E)$ be a tree. If $G$ has exactly two leaves (i.e., vertices of degree $1$), then $G$ is a path.
\end{lemma}

\begin{proof}
    Let $u,w$ denote the two leaves of $G$ (that is, $\deg_G(u)=\deg_G(w)=1$). Since $G$ is a tree and thus connected, it has no vertices of degree $0$. Hence, all non-leaves of $G$ must have degree at least $2$.
    That is, $\deg_G(v) \geq 2$ for all $v \in V \setminus \set{u,w}$.

    By the Handshaking Lemma, \begin{align*}
        2|E| = \sum_{v \in V}\deg_G(v) = \underbrace{\deg_G(u)}_{\textcolor{gray}{=1}}+\underbrace{\deg_G(w)}_{\textcolor{gray}{=1}}+\sum_{v \in V \setminus \set{u,w}}\deg_G(v) = 2+\sum_{v \in V \setminus \set{u,w}}\deg_G(v).
    \end{align*} Since every addend in the last sum is at least $2$ (that is, $\deg_G(v) \geq 2$ for all $v \in V \setminus \{u,w\}$), we get \begin{align}\label{twoedges}
        2|E|=2+\sum_{v \in V \setminus \set{u,w}}\deg_G(v) \geq 2+2\tup{|V|-2} = 2\tup{|V|-1}.
    \end{align} However, because $G$ is a tree, $|E|=|V|-1$. Hence, equality must hold everywhere in \eqref{twoedges}. In particular, this forces $\deg_G(v)=2$ for all $v \in V \setminus \set{u,w}$ (since $\deg_G(v) \geq 2$ for such vertices). 

    And so, every vertex in $G$ has degree $1$ or $2$, with exactly two vertices having degree $1$. Let $P$ be the unique path from $u$ to $w$ (which exists since $G$ is a tree). Suppose $P$ does not contain every vertex of $G$. Since $G$ is connected, there exists a vertex $x \notin P$ adjacent to some vertex $v \in P$. Since $\deg_G(u)=\deg_G(w)=1$, we have that $v \notin \set{u,w}$ (for otherwise $u$ or $w$ would have a neighbor outside $P$ in addition to its neighbor in $P$, contradicting the fact that it is a leaf). Hence, $v$ is an internal vertex of $P$. So, $v$ is adjacent to two vertices in $P$, in addition to being adjacent to $x \notin P$. Hence, $\deg_G(v) \geq 3$, contradicting the fact that every vertex in $V \setminus \{u,w\}$ has degree $2$. So, $P$ contains every vertex of $G$. Moreover, every edge of $G$ must be an edge of $P$ (since any further edge would yield a cycle). Hence, we conclude that $G$ is a path. In particular, the two leaves $u$ and $w$ are the endpoints of this path.
\end{proof}

\begin{lemma}\label{one_cycle_leaf}
    Let $G=(V,E)$ be a connected graph containing exactly one cycle.
    Assume that $G$ itself is not a cycle.
    Then, $G$ has at least one leaf.
\end{lemma}

\begin{proof}
    Suppose towards contradiction that $G$ has no leaves. Since $G$ is connected, no vertex in $G$ has degree $0$; and so, since $G$ has no leaves, every vertex has degree at least two.

    Let $C$ denote the edge set of the unique cycle in $G$. Choose an edge $e \in C$. Then $G-e$ is an acyclic connected graph (acyclic because we broke the only cycle of $G$; connected because removing an edge in a cycle cannot break connectivity),
    and hence a tree. So, $G-e$ has at least two leaves. Pick two, and call them $u$ and $w$. That is, \begin{align*}
        \deg_{G-e}(u)=\deg_{G-e}(w)=1.
    \end{align*} 
    
    Since every vertex in $G$ has degree at least two, the degrees of $u$ and $w$ must have decreased when $e$ was removed. And so, $u$ and $w$ are precisely the endpoints of $e$
    (and thus, $e$ is not a loop).

    We claim that $G-e$ has no other leaves. Indeed, if $x \notin \set{u,w}$ were a leaf of $G-e$, then removing $e$ would not affect the degree of $x$, since $x$ is not an endpoint of $e$. Hence, we would have \begin{align*}
        \deg_G(x)=\deg_{G-e}(x)=1,
    \end{align*} contradicting our assumption that $G$ has no leaves.

    And so, $G-e$ is a tree with exactly two leaves $u$ and $w$. By Lemma~\ref{two_leaves_path}, it follows that $G-e$ is a path. In particular, $u$ and $w$ are the endpoints of this path. Since $e = \set{u,w}$, adding edge $e$ joins the two endpoints of this path, so $G$ is a cycle. But this contradicts our assumption that $G$ itself was not a cycle. And so, we conclude that $G$ has at least one leaf.
\end{proof}

\begin{lemma}\label{odd_cycle_components}
    Let $G=(V,E)$ be a connected graph containing exactly one cycle, and suppose that cycle is odd. Then, the family of corresponding signless incidence vectors $\tup{\mathbf{b}_e : e \in E}$ is linearly independent in $\mathbb{R}^V$.
\end{lemma}

\begin{proof}
    First, if $G$ is itself an odd cycle (that is, $E$ is the edge set of an odd cycle), then the claim follows from Lemma~\ref{odd_cycles_independence}.

    Otherwise, suppose that $G$ is not a cycle. Then $G$ has at least one leaf by Lemma~\ref{one_cycle_leaf}. We now proceed exactly as in the proof of Lemma~\ref{forests_independence}, applying the same inductive reduction argument, with the only difference being that the base case is the odd cycle rather than a single edge
    (removing a leaf does not affect the cycles of our graph, nor its connectivity; thus, the assumptions of the lemma are preserved in this process).
\end{proof}

\begin{lemma}\label{ood_pseudoforest_independence}
    If $G=(V,E)$ is an odd pseudoforest, then the family of corresponding signless incidence vectors $\tup{\mathbf{b}_e : e \in E}$ is linearly independent in $\mathbb{R}^V$.
\end{lemma}

\begin{proof}
    Suppose $G=(V,E)$ is an odd pseudoforest. Then we can decompose $G$ into disjoint connected components \begin{align*}
        G_1,\ldots,G_m, \quad G_i = (V_i,E_i),
    \end{align*} where each $G_i$ contains at most one cycle, and this cycle is odd. That is, each $G_i$ is either a tree or contains exactly one cycle, which is odd.

    Consider a connected component $G_i$. If $G_i$ is a tree, then $G_i$ is also a forest; hence, the list of signless incidence vectors $\tup{\mathbf{b}_e : e \in E_i}$ is linearly independent in $\mathbb{R}^{V_i}$ by Lemma~\ref{forests_independence}. On the other hand, if $G_i$ contains exactly one cycle, and this cycle is odd, then the list of signless incidence vectors $\tup{\mathbf{b}_e : e \in E_i}$ is linearly independent in $\mathbb{R}^{V_i}$ by Lemma~\ref{odd_cycle_components}.

    In both cases, each component contributes a linearly independent family of signless incidence vectors. Since the connected components $G_1,\ldots,G_m$ are vertex-disjoint, their supports in $\mathbb{R}^V$ are disjoint. Extending each vector $\mathbf{b}_e \in \mathbb{R}^{V_i}$ to $\mathbb{R}^V$ by setting the entries outside $V_i$ to be zero, we may view each list $\tup{\mathbf{b}_e : e \in E_i}$ as a list of vectors in $\mathbb{R}^V$. Hence, the list \begin{align*}
        \tup{\mathbf{b}_e : e \in E} = \bigcup_{i=1}^m\tup{\mathbf{b}_e : e \in E_i}
    \end{align*} is a union of linearly independent lists with disjoint supports, so that $\tup{\mathbf{b}_e : e \in E}$ is linearly independent in $\mathbb{R}^V$
    (since any linear dependence relation $\sum_{e \in E} a_e \mathbf{b}_e = 0$ can be projected down on each $\mathbb{R}^{V_i}$ to yield a linear dependence relation $\sum_{e \in E_i} a_e \mathbf{b}_e = 0$, which entails $a_e = 0$ for all $e \in E_i$ by the linear independence of the family $\tup{\mathbf{b}_e : e \in E_i}$).
\end{proof}

We can now prove Theorem~\ref{pseudoforests_theorem}:

\begin{proof}[Proof of Theorem~\ref{pseudoforests_theorem}]
    Fix a finite graph $G=(V,E)$.
    Suppose $H=(V,F)$ is an odd pseudoforest in $G$. By Lemma~\ref{ood_pseudoforest_independence}, the list of signless incidence vectors $\tup{\mathbf{b}_e : e \in F}$ is linearly independent in $\mathbb{R}^V$. Hence, $F \in \mathcal{I}(X_G)$.

    And so, since $H=(V,F)$ was an arbitrary odd pseudoforest in $G$, we have that \begin{align}\label{pseudo_subset}
        \mathcal{I}(X_G)\supseteq \set{F \subseteq E : F \text{ is the edge set of an odd pseudoforest}}.
    \end{align} Consequently, \begin{align}
        \abs{\mathcal{I}(X_G)}  &\geq \abs{\set{F \subseteq E : F \text{ is the edge set of an odd pseudoforest}}} \nonumber \\
        &= (\# \text{ of odd pseudoforests in $G$})\label{pseudo_ineq}
    \end{align} (since odd pseudoforests -- like any spanning subgraphs -- are uniquely determined by their edge sets). 
    
    Hence, \begin{align*}
        &\tup{\# \text{ of degree sequences $\deg_H$ of spanning subgraphs $H \subseteq G$}} \\
        &= \abs{\Sigma(X_G)} \quad \textcolor{gray}{\text{(by \eqref{sum_equals_degs})}} \\
        &\geq \abs{\mathcal{I}(X_G)} \quad \textcolor{gray}{\text{(by Lemma~\ref{vector-list})}} \\
        &\geq (\#\text{ of odd pseudoforests in $G$}) \quad \textcolor{gray}{\text{(by \eqref{pseudo_ineq})}}.
    \end{align*} This proves Theorem~\ref{pseudoforests_theorem}.
\end{proof}

\subsection{Characterization of $\mathcal{I}(X_G)$}

We can actually strengthen the subset relation in \eqref{pseudo_subset} to an equality as follows:

\begin{theorem}\label{I=pseduo}
    Let $G=(V,E)$ be an undirected graph. Then,
    \[\mathcal{I}(X_G) = \set{F \subseteq E : F \text{ is the edge set of an odd pseudoforest}}.\]
\end{theorem}

This gives a complete characterization of the linearly independent subfamilies associated with the signless incidence vectors of a graph, showing that they correspond exactly to the edge sets of odd pseudoforests in $G$.

The proof of Theorem~\ref{I=pseduo} will rely on the following two lemmas:

\begin{lemma}\label{even_cycles_dependence}
    Let $G=(V,E)$ be an undirected graph. If $C \subseteq E$ is the edge set of an even cycle in $G$, then the corresponding signless incidence vectors $\tup{\mathbf{b}_e : e \in C}$ are linearly dependent  in $\mathbb{R}^V$.
\end{lemma}

\begin{remark}
    This is in contrast to Lemma~\ref{odd_cycles_independence}, in which the signless incidence vectors corresponding to an odd cycle are linearly \textit{independent} in $\mathbb{R}^{V}$.
\end{remark}

\begin{proof}[Proof of Lemma~\ref{even_cycles_dependence}]
    Suppose $|C|=2k$ for some $k \geq 1$. Label the vertices of the cycle $w_0,w_1,\ldots,w_{2k-1}$. Let $e_i \in C$ denote the edge joining vertices $w_i$ and $w_{i+1}$ for $i=0,1,\ldots,2k-2$, and let $e_{2k-1} \in C$ denote the edge joining vertices $w_{2k-1}$ and $w_0$.

    Consider the following sum: \begin{align*}
        \sum_{i=0}^{2k-1}\tup{-1}^{i}\mathbf{b}_{e_i}.
    \end{align*} Fix a vertex $w_i$ in the cycle. As in the proof of Lemma~\ref{odd_cycles_independence}, the only edges of $C$ incident to $w_i$ are $e_{i-1}$ and $e_i$ (with $e_{-1} :=e_{2k-1}$). Since each corresponding signless incidence vector has a $1$ in the $w_i$-coordinate, the $w_i$-coordinate of the above sum equals \begin{align*}
        (-1)^{i-1}+(-1)^i = 0.
    \end{align*} Since $w_i$ was arbitrary, we see that every coordinate of $\sum_{i=0}^{2k-1}(-1)^{i}\mathbf{b}_{e_i}$ is $0$ (since the coordinates corresponding to vertices not on the cycle $C$ are clearly $0$). Hence, \begin{align*}
        \sum_{i=0}^{2k-1}(-1)^{i}\mathbf{b}_{e_i} = \mathbf{0}
    \end{align*} Since the coefficients $\tup{-1}^i$ are not all zero, this is a nontrivial relation among the vectors $\tup{\mathbf{b}_e : e \in C}$. Thus, $\tup{\mathbf{b}_e : e \in C}$ is linearly dependent.
\end{proof}

\begin{lemma}\label{twoplus_cycles}
    Let $G=(V,E)$ be a finite graph, and let $H=(W,F)$ be a connected subgraph with $F \subseteq E$ and $W \subseteq V$. If $H$ contains at least two cycles, then the corresponding signless incidence vectors $\mathbf{b}_e$ for $e \in F$ are linearly dependent in $\mathbb{R}^{V}$.
\end{lemma}

\begin{proof}
Since $H$ is connected, it has a spanning tree whose edge set we denote by $T$. Hence, $|T| = |W|-1$. Now, every edge in $F \setminus T$ creates a unique cycle when added to $T$ (indeed, it creates a cycle because $T$ already has a path between its two endpoints; but it cannot create more than one cycle, since this path is unique). Since $H$ contains at least two cycles, it must therefore contain at least two distinct edges in $F \setminus T$. Thus, \begin{align*}
    |F| \geq |T|+2=(|W|-1)+2 = |W|+1 > |W|.
\end{align*}

Now, by abuse of notation, we let $\mathbb{R}^W$ denote the vector subspace of $\mathbb{R}^V$ that consists of all vectors supported on $W$ (that is, whose $v$-coordinate for each $v \in V \setminus W$ is $0$).
This subspace is $|W|$-dimensional, and all the $|F|$ many vectors $\tup{\mathbf{b}_e : e \in F}$ lie in it (since each edge $e \in F$ has both its endpoints in $W$).
However, $|F| > |W|$. Hence, the $|F|$ many vectors $\tup{\mathbf{b}_e : e \in F}$ are linearly dependent in $\mathbb{R}^W$ (since any family of more than $|W|$ vectors in the $|W|$-dimensional vector space $\mathbb{R}^W$ must necessarily be linearly dependent).
In other words, they are linearly dependent in the larger space $\mathbb{R}^V$ as well.
\end{proof}

We can now prove Theorem~\ref{I=pseduo}:

\begin{proof}
    \eqref{pseudo_subset} yields the ``$\supseteq$'' direction. Let us show the ``$\subseteq$" direction; that is, 
        \begin{align*}\mathcal{I}(X_G) \subseteq \set{F \subseteq E : F \text{ is the edge set of an odd pseudoforest}}.
    \end{align*} To do so, we will prove the contrapositive: if $F \subseteq E$ is not the edge set of an odd pseudoforest, then $\tup{\mathbf{b}_e : e \in F}$ is linearly dependent.

    Let $H=(V,F)$ be a spanning subgraph of $G$ such that $F \subseteq E$ is not the edge set of an odd pseudoforest. Then either $H$ contains an even cycle, or some connected component of $H$ contains at least two cycles.

    First, suppose $H$ contains an even cycle. Let $C \subseteq F$ be the edge set of the even cycle. By Lemma~\ref{even_cycles_dependence}, the family $\tup{\mathbf{b}_e : e \in C}$ is linearly dependent in $\mathbb{R}^V$. Since $C \subseteq F$ and linear dependence is inherited by supersets, the family $\tup{\mathbf{b}_e : e \in F}$ is linearly dependent in $\mathbb{R}^V$. That is, $F \notin \mathcal{I}(X_G)$.

    On the other hand, suppose some connected component of $H$ contains at least two cycles. Call this component $H' = (V',F')$. By Lemma~\ref{twoplus_cycles}, the family $\tup{\mathbf{b}_e : e \in F'}$ is linearly dependent in $\mathbb{R}^V$.
    Since $F' \subseteq F$ and since linear dependence is inherited by supersets, the family $\tup{\mathbf{b}_e : e \in F}$ is linearly dependent in $\mathbb{R}^V$. That is, $F \notin \mathcal{I}(X_G)$.

    And so, we see that if $F\subseteq E$ is not the edge set of an odd pseudoforest, then $F \notin \mathcal{I}(X_G)$. Equivalently, if $F \in \mathcal{I}(X_G)$, then $F$ is the edge set of an odd pseudoforest. And so,  \begin{align*}
        \mathcal{I}(X_G) \subseteq \set{F \subseteq E : F \text{ is the edge set of an odd pseudoforest}}.
    \end{align*} Together with (\ref{pseudo_subset}), we conclude that \begin{align*}
        \mathcal{I}(X_G) = \set{F \subseteq E: F \text{ is the edge set of an odd pseudoforest}}.
    \end{align*} This proves Theorem~\ref{I=pseduo}.
\end{proof}

\section*{Acknowledgements} The author thanks Darij Grinberg for introducing this project, for many helpful comments and guidance throughout its development, and for pointing out numerous relevant sources. The author thanks Alexander Postnikov for insights concerning the connections between degree sequences, score functions, and zonotopes (the latter of which ultimately did not play a role in the above work). The author also thanks Sergei Shteiner for pointing out the alternative proof in \cite{KleitmanWinston}. Finally, the author thanks V\'aclav Rozho\v{n} for drawing attention to the AI-generated proof of \eqref{conj} that has been presented above.

Parts of the writing in Sections~\ref{sec.nbp} and~\ref{sec.opf} are generated by GPT-5.5, although such material has been carefully proofread and partially rewritten by the author.



\nocite{*}
\bibliography{references}

\end{document}